\documentclass[11pt,twoside,reqno]{amsart}
\allowdisplaybreaks
\usepackage{amsmath,amstext,amssymb,epsfig,multicol,enumerate}
\usepackage{graphicx}
\usepackage{mathrsfs}
\calclayout
\makeatletter
\renewcommand{\@seccntformat}[1]{\bf\csname the#1\endcsname.}
\renewcommand{\section}{\@startsection{section}{1}
	\z@{.7\linespacing\@plus\linespacing}{.5\linespacing}
	{\normalfont\upshape\bfseries\centering}}
\renewcommand{\@biblabel}[1]{\@ifnotempty{#1}{#1.}}
\makeatother
\theoremstyle{plain}
\newtheorem{thm}{Theorem}[section]
\newtheorem{lem}[thm]{Lemma}
\newtheorem{prop}[thm]{Proposition}

\theoremstyle{definition}

\usepackage{cancel}
\usepackage[parfill]{parskip}
\usepackage[german]{varioref}
\usepackage[all]{xy}
\usepackage{color}

\def\A{{\mathcal A}}

\def \>{\succ}
\def \<{\prec}

\def\A{{\mathbb A}}
\def\C{{\mathbb C}}

\DeclareMathOperator{\Aut}{Aut}

\DeclareMathOperator{\id}{id}

\begin{document}	
\title[Ahmed Zahari ABDOU DAMDJI\textsuperscript{1}, Bouzid Mosbahi\textsuperscript{2}]{Compatible Pairs of Low-Dimensional Associative Algebras and Their Invariants}
	\author{Ahmed Zahari ABDOU DAMDJI\textsuperscript{1},  Bouzid Mosbahi\textsuperscript{2}}
\address{\textsuperscript{1}IRIMAS-Department of Mathematics, Faculty of Sciences, University of Haute Alsace, Mulhouse, France}
         \address{\textsuperscript{2}Department of Mathematics, Faculty of Sciences, University of Sfax, Sfax, Tunisia}
    
\email{\textsuperscript{1}abdou-damdji.ahmed-zahari@uha.fr}
	\email{\textsuperscript{2}mosbahi.bouzid.etud@fss.usf.tn}
         
	
	
	\date{\today}
\begin{abstract}
A \textit{compatible associative algebra} is a vector space endowed with two associative multiplication operations that satisfy a natural compatibility condition. In this paper, we investigate and classify compatible pairs of associative algebras of complex dimension less than four. Alongside these classifications, we systematically compute and analyze various algebraic invariants associated with them, including derivations, centroids, automorphism groups, quasi-centroids, Rota-Baxter operators, Nijenhuis operators, averaging operators, Reynolds operators, quasi-derivations, and generalized derivations.
\end{abstract}

\maketitle \section{ Introduction}\label{introduction}
The classification of algebraic structures and their invariants is an important and ongoing research area in mathematics and physics (see, for example, \cite{mak,maz,su,zah,akk,ki,ga,me}). Classical settings include associative, Lie, and Jordan algebras. Specifically, an associative algebra $(\A, \mu)$ consists of a vector space $\A$ equipped with a bilinear map $\mu:\A\otimes\A\rightarrow\A$, denoted $\mu(a\otimes b) = a\ast b$, that satisfies the associativity condition
$$(a\ast b)\ast c = a\ast(b\ast c)$$ for all $a, b, c \in\A$. In \cite{rak}, the authors classify associative algebras over the complex field with dimension less than five. On a related note, invariant structures such as derivations and centroids have been studied for various classes of algebras in numerous works \cite{hrb,bas,fii,leg,basd}. The algebra of derivations, for example, is useful in algebraic and geometric classification problems.

Other invariants include special types of linear maps from an algebra to itself, such as Rota-Baxter operators or Nijenhuis operators.

In the present paper, we are concerned with \textit{compatible} associative algebras, which are characterized by a pair of associative-algebraic structures over a common vector space that interact in a nice way. Specifically, a compatible associative algebra is a triple $(\A, \mu_1, \mu_2)$ for which $(\A, \mu_1)$ and $(\A, \mu_2)$ are both associative algebras that satisfy the compatibility axiom
\[\mu_2\circ(\mu_1\otimes \id) + \mu_1\circ(\mu_2\otimes \id) = \mu_1\circ(\id\otimes \mu_2) + \mu_2\circ(\id\otimes \mu_1)\] where $\id$ denotes the identity map on $\A$.
Letting $\bullet$ and $\ast$ denote the multiplications $\mu_1$ and $\mu_2$ respectively, the axiom becomes $$(a\bullet b)\ast c + (a\ast b)\bullet c =a\bullet (b\ast c) +a\ast(b\bullet c)$$ for $a, b, c\in \A$. We note that this identity is symmetric with respect to $\bullet$ and $\ast$.

The paper is structured as follows. We first classify compatible associative algebras of dimension less than four. We then classify derivations, centroids, automorphisms, and quasi-centroids for all algebras in these dimensions. For a selection of cases, we classify further invariants such as Rota-Baxter operators, Nijenhuis operators, averaging operators, Reynolds operators, quasi-derivations, and generalized derivations. The computations for our classifications were done using Maple and Mathematica. Throughout, we work over the complex field.

 \section{ Compatible Associative Algebras}
Consider a pair of compatible associative algebras $\A=(\A, \mu_1, \mu_2)$ and  $\A'=(\A', \mu_1', \mu_2')$. A linear transformation $\psi:\A\rightarrow \A'$ is termed a homomorphism if it satisfies $\psi\circ\mu_i=\mu_i'\circ(\psi\otimes\psi)$ for $i=1,2$. An homomorphism becomes an isomorphism if it is bijective. Two compatible associative algebras are termed isomorphic if there exists an isomorphism between them. We now recall the classification of complex associative algebras of dimensions 2, 3 and 4 (obtained in \cite{rak}). Utilizing these results, we classify complex compatible associative algebras of the same dimensions, which are presented as pairs of associative algebras.

To begin, consider an automorphism $\theta$ on a finite-dimensional complex compatible associative algebra $\A=(\A,\mu_1,\mu_2)$ where multiplications are denoted by $\mu_1(a\otimes b) = a\bullet b$ and $\mu_2(a\otimes b) = a\ast b$ for $a,b\in\A$. 
An automorphism $\theta$ on $\A$ is an isomorphism $\theta:\A\rightarrow \A$. The set of all automorphisms on $\A$ forms a group with respect to the composition, denoted by $\Aut(\A)$.
Let $\{e_1,e_2,\dots, e_n\}$ be a basis for $\A$, and denote the multiplications on this basis as: \begin{align*}
    e_i\bullet e_j = \alpha_{ij}^1e_1 + \alpha_{ij}^2e_2 + \cdots + \alpha_{ij}^ne_n, && e_i\ast e_j = \beta_{ij}^1e_1 + \beta_{ij}^2e_2 + \cdots + \beta_{ij}^ne_n
\end{align*} where, $\alpha_{ij}^k,\beta_{ij}^k\in \C$ are structure constants. Suppose \[\theta(e_i) = \theta_i^1e_1 + \theta_i^2e_2 + \cdots + \theta_i^ne_n\] for coefficients $\theta_i^1,\theta_i^2,\dots,\theta_i^n\in \C$.
Then, the equalities \begin{align*}
    \theta(e_i\bullet e_j) = \theta(e_i)\bullet \theta(e_j), && \theta(e_i\ast e_j) = \theta(e_i)\ast \theta(e_j)
\end{align*} yield
\begin{align*}
   \sum^{n}_{k=1} \alpha_{ij}^{k}\theta_k^r = 
   \sum^{n}_{q=1} \sum^{n}_{p=1}\theta_i^p\theta_j^q\alpha_{pq}^{r}, && 
     \sum^{n}_{k=1} \beta_{ij}^{k}\theta_k^r = 
   \sum^{n}_{q=1} \sum^{n}_{p=1}\theta_i^p\theta_j^q\beta_{pq}^{r}
\end{align*} for $r=1,2,\dots,n$.
 Solving this system of equations provides a classification of automorphisms on $\A$ using computer algebra methods.
\begin{lem}
The description of the automorphisms of every 2-dimensional compatible associative algebra is given below.

\[
\text{Aut}(\mathcal{A}_{2}^{1}, \mathcal{A}_{2}^{4}) =
\left\{
\begin{pmatrix}
1 & 0 \\
0 & \theta
\end{pmatrix}
\;\middle|\; \theta \in \mathbb{C}^*
\right\}.
\]
\end{lem}

\begin{thm}
The $2$-dimensional complex compatible associative algebras are given by the pairs

$\A_{2}^{1}$ :	
$\begin{array}{ll}
e_{1}\bullet e_{1}=e_{2},
\end{array}$
$ \A_{2}^{4}$ :	
$\begin{array}{ll}
e_{1}\ast e_{1}=e_{1},\,
 e_{1}\ast e_{2}=e_{2},\,
 e_{2}\ast e_{1}=e_{2}
\end{array}$
\end{thm}

\begin{proof}
    Let $\{e_1,e_2,\dots,e_n\}$ be a basis for a compatible associative algebra $\A$. Let \begin{align*}
        e_i\bullet e_j = \sum_{r=1}^n\alpha_{ij}^re_r, && e_i\ast e_j = \sum_{r=1}^n\beta_{ij}^re_r
    \end{align*} denote the structure constants on $\A$ for $i,j=1,2,\dots,n$. The identities \begin{align*}
        (a\bullet b)\ast c + (a\ast b)\bullet c &=a\bullet (b\ast c) +a\ast(b\bullet c), \\
        (a\bullet b)\bullet c &= a\bullet(b\bullet c)\\ (a\ast b)\ast c &= a\ast(b\ast c)
    \end{align*} yield a system of $n^3$ equations \[\sum_{r=1}^n\left(\alpha_{ij}^r\beta_{rk}^q + \beta_{ij}^r\alpha_{rk}^q\right) = \sum_{r=1}^n\left(\beta_{jk}^r\alpha_{ir}^q + \alpha_{jk}^r\beta_{ir}^q\right)\] \[\sum_{r=1}^n\alpha_{ij}^r\alpha_{rk}^q = \sum_{r=1}^n\alpha_{jk}^r\alpha_{ir}^q\]
    \[\sum_{r=1}^n\beta_{ij}^r\beta_{rk}^q = \sum_{r=1}^n\beta_{jk}^r\beta_{ir}^q\] for $i,j,k,q=1,2,\dots,n$ that we solve using computer algebra.
\end{proof}

\begin{lem}
The description of the automorphisms of every 3-dimensional compatible associative dialgebra is given below.
\begin{align*}
\text{Aut}(\mathcal{A}_3^1, \mathcal{A}_3^2) &= 
\left\{

\,\middle|\, \theta_1^1, \theta_1^2 \in \mathbb{C},\, \theta_1^1 \theta_1^2 \neq 0
\right\}.
\end{align*}


\end{lem}

\begin{thm}
The $3$-dimensional complex compatible associative algebras are given by the pairs
 \begin{enumerate}
\item $\A_{3}^{1}$ :	
$\begin{array}{ll}
e_{1}\bullet e_{3}=e_{2},\\
 e_{3}\bullet e_{1}=e_{2},
\end{array}$
$ \A_{3}^{2}$ :	
$\begin{array}{ll}
e_{1}\ast e_{3}=e_{2},\\
 e_{3}\ast e_{1}=\alpha e_{2}, \; \alpha \in \mathbb{C}\setminus \{1\}
\end{array}$
\item $\A_{3}^{1}$
$\begin{array}{ll}
e_{1}\bullet e_{3}=e_{2},\\
e_{3}\bullet e_{1}=e_{2}
\end{array}$
$ \A_{2}^{11}$ :	
$\begin{array}{llllll}
e_1\ast e_3=e_2,\\
e_2\ast e_3=e_2,\\
e_3\ast e_1=e_2,\\
\end{array}$
$\begin{array}{llllll}
e_3\ast e_2=e_2,\\
e_3\ast e_3=e_3
\end{array}$
\item $\A_3^{10}$ :	
$\begin{array}{llllll}
e_1\bullet e_3=e_1,\\
e_2\bullet e_3=e_2,\\
e_3\bullet e_1=e_1,\\
\end{array}$
$\begin{array}{llllll}
e_3\bullet e_2=e_2,\\
e_3\bullet e_3=e_3,\\
\end{array}\quad$
$\A_3^{12}  : $
$\begin{array}{llllll}
e_1\ast e_1=e_2,\\
e_1\ast e_3=e_1,\\
e_2\ast e_3=e_2,\\
\end{array}$
$\begin{array}{llllll}
 e_3\ast e_1=e_1,\\
 e_3\ast e_2=e_2,\\
 e_3\ast e_3=e_3
\end{array}$
\end{enumerate}
\end{thm}

\begin{lem}
The description of the automorphisms of every 4-dimensional compatible associative algebra is given below.

\[
\text{Aut}(\A_4^1, \A_4^3) = \left\{ 
\;\middle|\;
\begin{array}{l}
\theta_i^j \in \mathbb{C}, \\
\theta_2^1, \theta_2^2, \theta_2^3, \theta_2^4 \neq 0
\end{array}
\;\middle|\;
\begin{array}{l}
\theta_i^j \in \mathbb{C}, \\
\theta_2^1 \neq 0, \\
\theta_1^3 + \theta_2^3 \neq 0, \\
(\theta_1^4)^2 + (\theta_2^4)^2 \neq 0
\end{array}
\;\middle|\;
\begin{array}{l}
\theta_i^j \in \mathbb{C}, \\
\theta_2^1, \theta_2^2, \theta_2^3, \theta_2^4 \neq 0
\end{array}
\;\middle|\;
\begin{array}{l}
\theta_i^j \in \mathbb{C}, \\
\theta_2^1, \theta_2^2, \theta_2^3, \theta_2^4 \neq 0
\end{array}
\;\middle|\;
\begin{array}{l}
\theta_i^j \in \mathbb{C}, \\
\theta_1^1, \theta_1^2, \theta_1^3, \theta_1^4 \neq 0
\end{array}
\;\middle|\;
\begin{array}{l}
\theta_i^j \in \mathbb{C}, \\
\theta_1^1, \theta_2^2, \theta_2^3, \theta_2^4 \neq 0
\end{array}%
\middle| \theta_i^j \in \mathbb{C}, \theta_1^1, \theta_1^2, \theta_2^3, \theta_2^4 \neq 0 \right\}
\]

\end{lem}

\begin{thm}
The $4$-dimensional complex compatible associative algebras are given by the pairs
\begin{enumerate}
\item $\A_{4}^{1}$ :	
$\begin{array}{llllll}
e_{1}\bullet e_{2}=e_{3},\\
e_{2}\bullet e_{1}=e_{4}
\end{array}$
\quad
$\A_{4}^{3}  : $
$\begin{array}{llllll}
e_{1}\ast e_{2}=e_{3},\\
e_{2}\ast e_{1}=e_{4},
\end{array}$
$\begin{array}{llllll}
e_{2}\ast e_{2}=-e_{3}
\end{array}$

\item $\A_{4}^{1}$ :	
$\begin{array}{llllll}
e_{1}\bullet e_{2}=e_{3},\\
e_{2}\bullet e_{1}=e_{4}
\end{array}$
\quad
$ \A_{4}^{4} : $
$\begin{array}{llllll}
e_{1}\ast e_{2}=e_{3},\\
e_{2}\ast e_{1}=-e_{3},\\
\end{array}$
$\begin{array}{llllll}
e_{2}\ast  e_{2}=e_{4}
\end{array}$

\item $\A_{4}^{1}$ :	
$\begin{array}{llllll}
e_{1}\bullet e_{2}=e_{3},\\
e_{2}\bullet e_{1}=e_{4}
\end{array}$
\quad
$ \A_{4}^{6} : $
$\begin{array}{llllll}
e_{1}\ast e_{2}=e_{4},\\
e_{2}\ast e_{1}=\frac{1+\alpha}{1-\alpha}e_{4},\\
\end{array}$
$\begin{array}{llllll}
e_{2}\ast e_{2}=e_{3},\quad \alpha\neq1
\end{array}$

\item $\A_{4}^{1}$ :	
$\begin{array}{llllll}
e_{1}\bullet e_{2}=e_{3},\\
e_{2}\bullet e_{1}=e_{4}
\end{array}$
\quad
$ \A_{4}^{17} : $
$\begin{array}{llllll}
e_{1}\ast e_{1}=e_{4},\\
e_{1}\ast e_{2}=e_{3},\
\end{array}$
$\begin{array}{llllll}
e_{2}\ast e_{1}=-e_{3},\\
e_{2}\ast e_{2}=-2e_{3}+e_4
\end{array}$

\item $\A_{4}^{1}$ :	
$\begin{array}{llllll}
e_{1}\bullet e_{2}=e_{3},\\
e_{2}\bullet e_{1}=e_{4}
\end{array}$
\quad
$ \A_{4}^{18} : $
$\begin{array}{llllll}
 e_{1}\ast e_{1}=e_{4},\\
 e_{1}\ast e_{2}=e_{3},\\
\end{array}$
$\begin{array}{llllll}
e_{2}\ast e_{1}=-\alpha e_{3},\\
e_{2}\ast e_{2}=-e_{3}
\end{array}$

\item $\A_{4}^{1}$ :	
$\begin{array}{llllll}
e_{1}\bullet e_{2}=e_{3},\\
e_{2}\bullet e_{1}=e_{4}
\end{array}$
\quad
$ \A_{4}^{24}: $
$\begin{array}{llllll}
 e_{1}\ast e_{2}=e_{1},\\
 e_{2}\ast e_{1}=e_{1},\\
\end{array}$
$\begin{array}{llllll}
 e_{2}\ast e_{2}=e_{2}\\
e_{2}\ast e_{3}=e_{3},\\
\end{array}$
$\begin{array}{llllll}
e_{2}\ast e_{4}=e_{4}
\end{array}$

\item $\A_{4}^{2}$ :	
$\begin{array}{llllll}
 e_{1}\bullet e_{2}=e_{4},\\
 e_{3}\bullet e_{1}=e_{4}
\end{array}$
\quad
$ \A_{4}^{5}: $
$\begin{array}{llllll}
  e_{1}\ast e_{2}=e_{4},\\
  e_{1}\ast e_{2}=-e_{4},
\end{array}$
$\begin{array}{llllll}
\,e_{3}\ast e_{3}=e_{4}
\end{array}$

\item $\A_{4}^{2}$ :	
$\begin{array}{llllll}
 e_{1}\bullet e_{2}=e_{4},\\
 e_{3}\bullet e_{1}=e_{4}
\end{array}$
\quad
$ \A_{4}^{10}: $
$\begin{array}{llllll}
  e_{1}\ast  e_{1}=e_{3},\\
  e_{1}\ast  e_{3}=e_{4},\\ 
\end{array}$
$\begin{array}{llllll}
 e_{2}\ast  e_{2}=-e_{4},\\
e_{3}\ast  e_{1}=e_{4}
\end{array}$

\item $\A_{4}^{2}$ :	
$\begin{array}{llllll}
 e_{1}\bullet  e_{2}=e_{4},\\
 e_{3}\bullet  e_{1}=e_{4}
\end{array}$
\quad
$ \A_{4}^{16}: $
$\begin{array}{llllll}
e_{1}\ast  e_{1}=e_{4},\\
e_{1}\ast  e_{2}=e_{4},\\
e_{2}\ast  e_{1}=\alpha e_{4}\\
e_{3}\ast  e_{3}=e_{4}
\end{array}$

\item $\A_{4}^{2}$ :	
$\begin{array}{llllll}
 e_{1}\bullet e_{2}=e_{4},\\
 e_{3}\bullet e_{1}=e_{4}
\end{array}$
\quad
$ \A_{4}^{25}: $
$\begin{array}{llllll}
e_1\ast e_2=e_4,\\
e_1\ast e_3=-e_3,\\ 
e_2\ast e_1=-e_4,\\
\end{array}$
$\begin{array}{llllll}
e_2\ast e_2=e_4,\\
e_4\ast e_1=e_4,
\end{array}$

\item $\A_{4}^{2}$ :	
$\begin{array}{llllll}
 e_{1}\bullet e_{2}=e_{4},\\
 e_{3}\bullet e_{1}=e_{4}
\end{array}$
\quad
$ \A_{4}^{28}: $
$\begin{array}{llllll}
 e_1\ast e_1=e_4,\\
 e_1\ast e_2=\alpha e_4,\\
  e_2\ast e_1=-\alpha e_4, \\
\end{array}$
$\begin{array}{llllll}
 e_2\ast e_2=e_4,\\
 e_3\ast e_3=e_4,\\
\end{array}$

\item $\A_{4}^{3}$ :	
$\begin{array}{llllll}
 e_{1}\bullet e_{2}=e_{3}, \\
 e_{2}\bullet e_{1}=e_{4},\\
e_{2}\bullet e_{2}=-e_{3},\\
\end{array}$
\quad
$  \A_{4}^{4}: $
$\begin{array}{llllll}
e_{1}\ast e_{2}=e_{3},\\
e_{2}\ast e_{1}=-e_{3},\\
e_{2}\ast e_{3}=e_{4}
\end{array}$

\item $\A_{4}^{3}$ :	
$\begin{array}{llllll}
  e_{1}\bullet e_{2}=e_{3}, \\
 e_{2}\bullet e_{1}=e_{4},\\
e_{2}\bullet e_{2}=-e_{3},\\
\end{array}$
\quad
$ \A_{4}^{6}: $
$\begin{array}{llllll}
 e_1\ast e_2=e_4,\\
 e_2\ast e_1=\frac{\alpha+1}{\alpha-1}e_4\\
e_2\ast e_2=e_3
\end{array}$

\item $\A_{4}^{3}$ :	
$\begin{array}{llllll}
 e_{1}\bullet e_{2}=e_{3}, \\
 e_{2}\bullet e_{1}=e_{4},\\
e_{2}\bullet e_{2}=-e_{3},\\
\end{array}$
\quad
$ \A_{4}^{17}: $
$\begin{array}{llllll}
  e_1\ast  e_1=e_4,\\
  e_1\ast  e_2=e_3,\\
 e_2\ast  e_1=-e_3,\\
 e_2\ast  e_2=-2e_3+e_4
\end{array}$

\item $\A_{4}^{3}$ :	
$\begin{array}{llllll}
 e_{1}\bullet e_{2}=e_{3}, \\
 e_{2}\bullet e_{1}=e_{4},\\
e_{2}\bullet e_{2}=-e_{3},\\
\end{array}$
\quad
$ \A_{4}^{18}: $
$\begin{array}{llllll}
  e_1\ast e_1=e_4,\\
  e_1\ast e_2=e_3,\\
 e_2\ast e_1=-\alpha e_4,\\
e_2\ast e_2=-e_3
\end{array}$

\item $\A_{4}^{4}$ :	
$\begin{array}{llllll}
e_{1}\bullet e_{2}=e_{3},\\
e_{2}\bullet e_{1}=-e_{3},\\
e_{2}\bullet e_{2}=e_{4}
\end{array}$
\quad
$ \A_{4}^{6} : $
$\begin{array}{llllll}
 e_1\ast e_2=e_4,\\
 e_2\ast e_1=\frac{\alpha+1}{\alpha-1}e_4\\
e_2\ast e_2=e_3
\end{array}$

\item $\A_{4}^{4}$ :	
$\begin{array}{llllll}
e_{1}\bullet e_{2}=e_{3},\\
e_{2}\bullet e_{1}=-e_{3},\\
e_{2}\bullet e_{2}=e_{4}
\end{array}$
\quad
$ \A_{4}^{17}: $
$\begin{array}{llllll}
  e_1\ast e_1=e_4,\\
  e_1\ast e_2=e_3,\\
 e_2\ast e_1=-e_3,\\
 e_2\ast e_2 =-2e_3+e_4
\end{array}$

\item $\A_{4}^{4}$ :	
$\begin{array}{llllll}
e_{1}\bullet e_{2}=e_{3},\\
e_{2}\bullet e_{1}=-e_{3},\\
e_{2}\bullet e_{2}=e_{4}
\end{array}$
\quad
$ \A_{4}^{18}: $
$\begin{array}{llllll}
  e_1\ast e_1=e_4,\\
  e_1\ast e_2=e_3,\\
 e_2\ast e_1=-\alpha e_4,\\
e_2\ast e_2=-e_3
\end{array}$

\item $\A_{4}^{5}$ :	
$\begin{array}{llllll}
e_1\bullet e_2=e_4,\\
e_2\bullet e_1=-e_4,\\
e_3\bullet e_3=e_4
\end{array}$
\quad
$ \A_{4}^{25} : $
$\begin{array}{llllll}
e_1\ast e_2=e_4,\\
e_1\ast e_3=e_4,\\ 
e_2\ast e_1=-e_4,\\
e_2\ast e_2=e_4,\\
e_3\ast e_1=e_4,
\end{array}$

\item $\A_{4}^{5}$ :	
$\begin{array}{llllll}
e_1\bullet e_2=e_4,\\
e_2\bullet e_1=-e_4,\\
e_3\bullet e_3=e_4
\end{array}$
\quad
$ \A_{4}^{28} : $
$\begin{array}{llllll}
e_1\ast e_1=e_4,\\
e_1\ast e_2=\alpha e_4,\\ 
e_2\ast e_1=-\alpha e_4,\\
e_2\ast e_2=e_4,\\
e_3\ast e_3=e_4
\end{array}$

\item $\A_{4}^{6}$ :	
$\begin{array}{llllll}
 e_1\bullet e_2=e_4,\\
 e_2\bullet e_1=\frac{\alpha+1}{\alpha-1}e_4\\
e_2\bullet e_2=e_3
\end{array}$
\quad
$ \A_{4}^{17}: $
$\begin{array}{llllll}
  e_1\ast e_1=e_4,\\
  e_1\ast e_2=e_3,\\
 e_2\ast e_1=-e_3,\\
 e_2\ast e_2=-2e_3+e_4
\end{array}$

\item $\A_{4}^{6}$ :	
$\begin{array}{llllll}
 e_1\bullet e_2=e_4,\\
 e_2\bullet e_1=\frac{\alpha+1}{\alpha-1}e_4\\
e_2\bullet e_2=e_3
\end{array}$
\quad
$ \A_{4}^{18}: $
$\begin{array}{llllll}
  e_1\ast e_1=e_4,\\
  e_1\ast e_2=e_3,\\
 e_2\ast e_1=-\alpha e_4,\\
e_2\ast e_2=-e_3
\end{array}$

\item $\A_{4}^{10}$ :	
$\begin{array}{llllll}
  e_{1}\bullet e_{1}=e_{3},\\
  e_{1}\bullet e_{3}=e_{4},\\ 
 e_{2}\bullet e_{2}=-e_{4},\\
e_{3}\bullet e_{1}=e_{4}
\end{array}$
\quad
$ \A_{4}^{25}: $
$\begin{array}{llllll}
e_1\ast  e_2=e_4,\\
e_1\ast  e_3=-e_3,\\ 
e_2\ast  e_1=-e_4,\\
e_2\ast  e_2=e_4,\\
e_4\ast  e_1=e_4,
\end{array}$

\item $\A_{4}^{11}$ :	
$\begin{array}{llllll}
e_1\bullet e_1=e_4,\\  
e_1\bullet e_4=-e_3,\\ 
 e_2\bullet e_1=e_3,\\
 e_4\bullet e_1=-e_3
\end{array}$
\quad
$ \A_{4}^{23}(\alpha): $
$\begin{array}{llllll}
e_1\ast  e_1=e_4,\\
e_1\ast  e_4=-e_3,\\ 
e_2\ast  e_1= e_3,\\
e_2\ast  e_2=e_3,\\
e_4\ast  e_1=-e_3,\\
 \end{array}$

\item $\A_{4}^{16}(\alpha)$ :	
$\begin{array}{llllll}
e_1\bullet e_1=e_4,\\
e_1\bullet e_2=e_4,\\   
e_2\bullet e_1=\alpha e_4,\\
e_3\bullet e_3=e_4
\end{array}$
\quad
$ \A_{4}^{25}: $
$\begin{array}{llllll}
e_1\ast  e_2=e_4,\\
e_1\ast  e_3=e_4,\\ 
e_2\ast  e_1=-e_4,\\
e_2\ast  e_2=e_4,\\
e_3\ast  e_1=e_4,
\end{array}$

\item $\A_{4}^{16}(\alpha)$ :	
$\begin{array}{llllll}
e_1\bullet e_1=e_4,\\
e_1\bullet e_2=e_4,\\   
e_2\bullet e_1=\alpha e_4,\\
e_3\bullet e_3=e_4
\end{array}$
\quad
$ \A_{4}^{28}: $
$\begin{array}{llllll}
e_1\ast e_1=e_4,\\
e_1\ast e_2=\alpha e_4,\\ 
e_2\ast e_1=-\alpha e_4,\\
e_2\ast e_2=e_4,\\
e_3\ast e_3=e_4
\end{array}$

\item $\A_{4}^{17}$ :	
$\begin{array}{llllll}
e_1\bullet e_1=e_4,\\
e_1\bullet e_2=e_3,\\   
e_2\bullet e_2=-e_3,\\
e_2\bullet e_2=-2e_3+e_4
\end{array}$
\quad
$ \A_{4}^{18}: $
$\begin{array}{llllll}
e_1\ast e_1=e_4,\\
e_1\ast e_2=e_3,\\ 
e_2\ast e_1=-\alpha e_4,\\
 e_2\ast e_2=-e_3,\\
\end{array}$

\item $\A_{4}^{25}$ :	
$\begin{array}{llllll}
e_1\bullet e_2=e_4,\\
e_1\bullet e_3=-e_3,\\ 
e_2\bullet e_1=-e_4,\\
e_2\bullet e_2=e_4,\\
e_4\bullet e_1=e_4,
\end{array}$
\quad
$ \A_{4}^{28}: $
$\begin{array}{llllll}
e_1\ast e_1=e_4,\\
e_1\ast e_2=\alpha e_4,\\ 
e_2\ast e_1=-\alpha e_4,\\
e_2\ast e_2=e_4,\\
e_3\ast e_3=e_4
\end{array}$

\item $\A_{4}^{25}$ :	
$\begin{array}{llllll}
e_1\bullet e_2=e_4,\\
e_1\bullet e_3=-e_3,\\ 
e_2\bullet e_1=-e_4,\\
e_2\bullet e_2=e_4,\\
e_4\bullet e_1=e_4,
\end{array}\quad$
$ \A_{4}^{33}: $
$\begin{array}{llllll}
e_1\ast e_1=e_2,\\
e_1\ast e_2=e_3,\\ 
e_1\ast e_3=e_4,\\
\end{array}$
$\begin{array}{llllll}
e_2\ast e_1=e_3,\\
e_2\ast e_2=e_4,\\
e_3\ast e_1=e_4,\\
\end{array}$

\item $\A_{4}^{33}$ :	
$\begin{array}{llllll}
e_1\bullet e_1=e_2,\\
e_1\bullet e_2=e_3,\\ 
e_1\bullet e_3=e_4,\\
e_2\bullet e_1=e_3,\\
e_2\bullet e_2=e_4,\\
e_3\bullet e_1=e_4,\\
\end{array}\quad$
$ \A_{4}^{45}: $
$\begin{array}{llllll}
e_1\ast e_1=e_1,\\
e_1\ast e_2=e_2,\\
e_1\ast e_3=e_3,\\  
 e_1\ast e_4=e_4,\\
  e_2\ast e_1=e_2,\\
\end{array}\quad$
$\begin{array}{llllll}
 e_2\ast e_4=e_3,\\
 e_2\ast e_3=e_4,\\
 e_3\ast e_1=e_3,\\
  e_3\ast e_2=e_4,\\
 e_4\ast e_1=e_1,\\
\end{array}$

\item $\A_{4}^{37}$ :	
$\begin{array}{llllll}
e_1\bullet e_1=e_1,\\
e_1\bullet e_2=e_2,\\ 
e_1\bullet e_3=e_3,\\
e_2\bullet e_1=e_2,\\
e_3\bullet e_1=e_3,\\
e_4\bullet e_1=e_4,\\
\end{array}\quad$
$ \A_{4}^{39}: $
$\begin{array}{llllll}
e_1\ast e_1=e_1,\\
e_1\ast e_2=e_2,\\
e_1\ast e_3=e_3,\\  
e_2\ast e_1=e_2,\\
e_2\ast e_2=e_3,\\
 e_3\ast e_1=e_3,\\
 e_4\ast e_1=e_1,\\
\end{array}$

\item $\A_{4}^{40}$ :	
$\begin{array}{llllll}
e_1\bullet e_1=e_1,\\
e_1\bullet e_2=e_2,\\
e_1\bullet e_3=e_3, \\
e_1\bullet e_4=e_4, 
\end{array}\quad$
$\begin{array}{llllll}
e_2\bullet e_1=e_2,\\
e_3\bullet e_1=e_3,\\
e_4\bullet e_1=e_4,\\
\end{array}$
\quad
$\A_{4}^{43}$ :	
$\begin{array}{llllll}
e_1\ast e_1=e_1,\\
e_1\ast e_2=e_2,\\
e_1\ast e_3=e_3\\ 
e_1\ast e_4=e_4\\ 
\end{array}\quad$
$\begin{array}{llllll}
e_2\ast e_1=e_2,\\
e_2\ast e_2=e_4,\\
e_3\ast e_1=e_3,\\
e_4\ast e_1=e_4,\\
\end{array}$

\item $\A_{4}^{40}$ :	
$\begin{array}{llllll}
e_1\bullet e_1=e_1,\\
e_1\bullet e_2=e_2,\\
e_1\bullet e_3=e_3, \\
e_1\bullet e_4=e_4, 
\end{array}\quad$
$\begin{array}{llllll}
e_2\bullet e_1=e_2,\\
e_3\bullet e_1=e_3,\\
e_4\bullet e_1=e_4,\\
\end{array}$
\quad
$ \A_{4}^{44}: $
$\begin{array}{llllll}
e_1\ast e_1=e_1,\\
e_1\ast e_2=e_2,\\
e_1\ast e_3=e_3,\\ 
e_1\ast e_4=e_4,\\
e_2\ast e_1=e_2,\\ 
\end{array}\quad$
$\begin{array}{llllll}
e_2\ast e_3=\alpha e_4,\\
e_3\ast e_1=e_3,\\
e_3\ast e_2=e_4,\\
e_4\ast e_2=e_4,\\
\end{array}$

\item $\A_{4}^{40}$ :	
$\begin{array}{llllll}
e_1\bullet e_1=e_1,\\
e_1\bullet e_2=e_2,\\
e_1\bullet e_3=e_3, \\
e_1\bullet e_4=e_4, 
\end{array}\quad$
$\begin{array}{llllll}
e_2\bullet e_1=e_2,\\
e_3\bullet e_1=e_3,\\
e_4\bullet e_1=e_4\\
\end{array}$
\quad
$ \A_{4}^{45}: $
$\begin{array}{llllll}
e_1\ast e_1=e_1,\\
e_1\ast e_2=e_2,\\
e_1\ast e_3=e_3,\\  
 e_1\ast e_4=e_4,\\
  e_2\ast e_1=e_2,\\
\end{array}\quad$
$\begin{array}{llllll}
 e_2\ast e_4=e_3,\\
 e_2\ast e_3=e_4,\\
 e_3\ast e_1=e_3,\\
  e_3\ast e_2=e_4,\\
 e_4\ast e_1=e_1,\\
\end{array}$

\item $\A_{4}^{40}$ :	
$\begin{array}{llllll}
e_1\bullet e_1=e_1,\\
e_1\bullet e_2=e_2,\\
e_1\bullet e_3=e_3, \\
e_1\bullet e_4=e_4, 
\end{array}\quad$
$\begin{array}{llllll}
e_2\bullet e_1=e_2,\\
e_3\bullet e_1=e_3,\\
e_4\bullet e_1=e_4,\\
\end{array}$
\quad
$ \A_{4}^{46}: $
$\begin{array}{llllll}
e_1\ast e_1=e_1,\\
e_1\ast e_2=e_3,\\
e_1\ast e_3=e_4,\\
e_1\ast e_4=e_4,\\
e_2\ast e_1=e_2
\end{array}\quad$
$\begin{array}{llllll}
e_2\ast e_2=-e_4,\\
e_2\ast e_3=-e_4,\\
e_3\ast e_1=e_3,\\
e_3\ast e_2=e_4\\
e_4\ast e_1=e_4
\end{array}$

\item 
$\A_{4}^{43}$ :	
$\begin{array}{llllll}
e_1\bullet e_1=e_1,\\
e_1\bullet e_2=e_2,\\
e_1\bullet e_3=e_3\\ 
e_1\bullet e_4=e_4\\ 
\end{array}\quad$
$\begin{array}{llllll}
e_2\bullet e_1=e_2,\\
e_2\bullet e_2=e_4,\\
e_3\bullet e_1=e_3,\\
e_4\bullet e_1=e_4,\\
\end{array}$
\quad 
$\A_{4}^{44}: $
$\begin{array}{llllll}
e_1\ast e_1=e_1,\\
e_1\ast e_2=e_2,\\
e_1\ast e_3=e_3,\\ 
e_1\ast e_4=e_4,\\
e_2\ast e_1=e_2,\\ 
\end{array}\quad$
$\begin{array}{llllll}
e_2\ast e_3=\alpha e_4,\\
e_3\ast e_1=e_3,\\
e_3\ast e_2=e_4,\\
e_4\ast e_2=e_4,\\
\end{array}$

\item $\A_{4}^{43}$ :	
$\begin{array}{llllll}
e_1\bullet e_1=e_1,\\
e_1\bullet e_2=e_2,\\
e_1\bullet e_3=e_3\\ 
e_1\bullet e_4=e_4\\ 
\end{array}\quad$
$\begin{array}{llllll}
e_2\bullet e_1=e_2,\\
e_2\bullet e_2=e_4,\\
e_3\bullet e_1=e_3,\\
e_4\bullet e_1=e_4,\\
\end{array}$
\quad
$ \A_{4}^{45}: $
$\begin{array}{llllll}
e_1\ast e_1=e_1,\\
e_1\ast e_2=e_2,\\
e_1\ast e_3=e_3,\\  
 e_1\ast e_4=e_4,\\
  e_2\ast e_1=e_2,\\
\end{array}\quad$
$\begin{array}{llllll}
 e_2\ast e_4=e_3,\\
 e_2\ast e_3=e_4,\\
 e_3\ast e_1=e_3,\\
  e_3\ast e_2=e_4,\\
 e_4\ast e_1=e_1,\\
\end{array}$

\item $\A_{4}^{43}$ :	
$\begin{array}{llllll}
e_1\bullet e_1=e_1,\\
e_1\bullet e_2=e_2,\\
e_1\bullet e_3=e_3\\ 
e_1\bullet e_4=e_4\\ 
\end{array}\quad$
$\begin{array}{llllll}
e_2\bullet e_1=e_2,\\
e_2\bullet e_2=e_4,\\
e_3\bullet e_1=e_3,\\
e_4\bullet e_1=e_4,\\
\end{array}$
\quad
$ \A_{4}^{46}: $
$\begin{array}{llllll}
e_1\ast e_1=e_1,\\
e_1\ast e_2=e_3,\\
e_1\ast e_3=e_4,\\
e_1\ast e_4=e_4,\\
e_2\ast e_1=e_2
\end{array}\quad$
$\begin{array}{llllll}
e_2\ast e_2=-e_4,\\
e_2\ast e_3=-e_4,\\
e_3\ast e_1=e_3,\\
e_3\ast e_2=e_4\\
e_4\ast e_1=e_4
\end{array}$

\item $\A_{4}^{44}$ :	
$\begin{array}{llllll}
e_1\bullet e_1=e_1,\\
e_1\bullet e_2=e_2,\\
e_1\bullet e_3=e_3,\\ 
e_1\bullet e_4=e_4,\\
e_2\bullet e_1=e_2,\\ 
\end{array}\quad$
$\begin{array}{llllll}
e_2\bullet e_3=\alpha e_4,\\
e_3\bullet e_1=e_3,\\
e_3\bullet e_2=e_4,\\
e_4\bullet e_2=e_4,\\
\end{array}$
\quad
$ \A_{4}^{46}: $
$\begin{array}{llllll}
e_1\ast e_1=e_1,\\
e_1\ast e_2=e_3,\\
e_1\ast e_3=e_4,\\
e_1\ast e_4=e_4,\\
e_2\ast e_1=e_2
\end{array}\quad$
$\begin{array}{llllll}
e_2\ast e_2=-e_4,\\
e_2\ast e_3=-e_4,\\
e_3\ast e_1=e_3,\\
e_3\ast e_2=e_4\\
e_4\ast e_1=e_4
\end{array}$
\end{enumerate}
\end{thm}

\section{ Algebraic Invariants}
The aim of this section is to classify a selection of invariants for compatible associative algebras that are analogous to well-known operators from other algebraic settings. Most of these take the form of linear operators that are characterized by certain identities. In particular, let $(\A,\mu_1,\mu_2)$ be a compatible associative algebra with multiplications denoted by $\mu_1(a\otimes b) = a\bullet b$ and $\mu_2(a\otimes b) = a\ast b$ for $a,b\in \A$. In the following definitions, we use  $\bullet$ and $\ast$ and work with elements $a,b\in \A$.

\begin{enumerate}
    \item A derivation is a linear map $d:\A\rightarrow \A$ that satisfies\\
    $d(a\bullet b) = d(a)\bullet b + a\bullet d(b)$ and $d(a\ast b) = d(a)\ast b + a\ast d(b)$  
    
	\item A linear map $d:\A\rightarrow \A$ is called a quasi-derivation if there exists another linear map $d':\A\rightarrow \A$ such that $d'(a\bullet b) = d(a)\bullet b + a\bullet d(b)$ and $d'(a\ast b) = d(a)\ast b + a\ast d(b)$

 \item A linear map $d:\A\rightarrow \A$ is called a generalized derivation if there exist linear maps $d',d'':\A\rightarrow \A$ such that $d''(a\bullet b) = d(a)\bullet b + a\bullet d'(b)$ and $d''(a\ast b) = d(a)\ast b + a\ast d'(b)$

    \item A centroid is a linear map $\eta :\A\xrightarrow{} \A$ that satisfies\\ $\eta(a\bullet b) = \eta(a)\bullet b = a\bullet \eta(b)$ and $\eta(a\ast b) = \eta(a)\ast b = a\ast \eta(b)$

    \item A quasi-centroid is a linear map $\delta:\A\xrightarrow{} \A$ that satisfies $\delta(a)\bullet b = a\bullet \delta(b)$ and $\delta(a)\ast b = a\ast \delta(b)$ 

    \item A Rota-Baxter operator is a linear map $R:\A\xrightarrow{} \A$ that satisfies\\ $R(a)\bullet R(b) = R(R(a)\bullet b + a\bullet R(b))$ and  $R(a)\ast R(b) = R(R(a)\ast b + a\ast R(b))$

    \item A Nijenhuis operator is a linear map $N:\A\xrightarrow{} \A$ such that\\ $N(a)\bullet N(b) = N(N(a)\bullet b)+a\bullet N(b)-N(a\bullet b))$ and  $N(a)\ast N(b) = N(N(a)\ast b)+a\ast N(b)-N(a\ast b))$

    \item An averaging operator is a linear map $\chi:\A\xrightarrow{} \A$ such that\\
    $\chi(\chi(a)\bullet b)=\chi(a)\bullet \chi(b)=\chi(a\bullet \chi(b)))$ and  
    $\chi(\chi(a)\ast b)=\chi(a)\ast \chi(b)=\chi(a\ast \chi(b)))$

     \item A Reynolds operator is a linear map $\xi:\A\xrightarrow{} \A$ such that\\ 
		$\xi(a\bullet b)  =\xi(\xi(a)\bullet b+a\bullet \xi(b)-\xi(a)\bullet \xi(b))$ and $\xi(a\ast b)  =\xi(\xi(a)\ast  b+a\ast \xi(b)-\xi(a)\ast \xi(b))$
\end{enumerate}
We now describe the process for classifying these invariants.

 We use a similar process to classify the other invariants. As another example, let $d$ be a derivation on $\A$ with \[d(e_i) = d_{i}^1e_1 + d_{i}^2e_2 + \cdots + d_{i}^ne_n\] for $d_i^1,d_i^2,\dots,d_i^n\in\C$. The equalities \begin{align*}
    d(e_i\bullet e_j) = d(e_i)\bullet e_j + e_i\bullet d(e_j), && d(e_i\ast e_j) = d(e_i)\ast e_j + e_i\ast d(e_j)
\end{align*} yield
\begin{align*}
    \sum_{k=1}^n \alpha_{ij}^{k}d_k^r = \sum_{k=1}^n \left( d_i^k\alpha_{kj}^{r} +d_j^k\alpha_{ik}^r \right)\!, && \sum_{k=1}^n \beta_{ij}^{k}d_k^r = \sum_{k=1}^n \left( d_i^k\beta_{kj}^{r} +d_j^k\beta_{ik}^r \right)
\end{align*} for $i,j,r=1,2,\dots,n$.

We use a similar process to classify the other invariants. Let $\eta$ be a centroid on $\A$ with \[\eta(e_i) = \eta_{i}^1e_1 + \eta_{i}^2e_2 + \cdots + \eta_{i}^ne_n\] for $\eta_i^1,\eta_i^2,\dots,\eta_i^n\in\C$. The equalities \begin{align*}
    \eta(e_i\bullet e_j) = \eta(e_i)\bullet e_j = e_i\bullet \eta(e_j), && \eta(e_i\ast e_j) = \eta(e_i)\ast e_j = e_i\ast \eta(e_j)
\end{align*} yield
\begin{align*}
    \sum_{k=1}^n \alpha_{ij}^{k}\eta_k^r = \sum_{k=1}^n \eta_i^k\alpha_{kj}^{r} = \sum_{k=1}^n \eta_j^k\alpha_{ik}^r, && \sum_{k=1}^n\beta_{ij}^{k}\eta_k^r = \sum_{k=1}^n \eta_i^k\beta_{kj}^{r} = \sum_{k=1}^n \eta_j^k\beta_{ik}^r
\end{align*} for $i,j,r=1,2,\dots,n$.

We use a similar process to classify the other invariants. Let $\delta$ be a quasi-centroid on $\A$ with \[ \delta(e_i) =  \delta_{i}^1e_1 +  \delta_{i}^2e_2 + \cdots +  \delta_{i}^ne_n\] for $ \delta_i^1, \delta_i^2,\dots, \delta_i^n\in\C$. The equalities
\begin{align*}
\delta(e_i)\bullet e_j = e_i\bullet \delta(e_j), && \delta(e_i)\ast e_j = e_i\ast  \delta(e_j)
\end{align*} yield
\begin{align*}
 \sum_{k=1}^n\delta_i^k\alpha_{kj}^{r} = \sum_{k=1}^n\delta_j^k\alpha_{ik}^r, && \sum_{k=1}^n \delta_i^k\beta_{kj}^{r} = \sum_{k=1}^n\delta_j^k\beta_{ik}^r
\end{align*} for $i,j,r=1,2,\dots,n$.

We use a similar process to classify the other invariants. Let $R$ be a Rota-Baxter operator on $\A$ with \[R(e_i) = R_{i}^1e_1 + R_{i}^2e_2 + \cdots +R_{i}^ne_n\] for $ R_i^1, R_i^2,\dots, R_i^n\in\C$. The equalities
\begin{align*}
R(e_i)\bullet R(e_j) = R(R(e_i)\bullet e_j + e_i\bullet R(e_j)), && R(e_i)\ast R(e_j) = R(R(e_i)\ast e_j + e_i\ast R(e_j))
\end{align*} yield
\begin{align*}
 \sum_{k=1}^n\sum_{p=1}^nR_i^kR_j^p\alpha_{kp}^{q} = \sum_{k=1}^n\sum_{p=1}^n\Bigg(R_i^k\alpha_{kj}^pR_p^q+R_j^k\alpha_{ik}^pR_p^q\Bigg), &&  \sum_{k=1}^n\sum_{p=1}^nR_i^kR_j^p\beta_{kp}^{q} = \sum_{k=1}^n\sum_{p=1}^n\Bigg(R_i^k\beta_{kj}^pR_p^q+R_j^k\beta_{ik}^pR_p^q\Bigg)
\end{align*} for $i,j,q=1,2,\dots,n$.

We use a similar process to classify the other invariants. Let $N$ be a Nijenhuis operator on $\A$ with \[N(e_i) = N_{i}^1e_1 + N_{i}^2e_2 + \cdots +N_{i}^ne_n\] for $ N_i^1, N_i^2,\dots, N_i^n\in\C$. The equalities
\begin{align*}
N(e_i)\bullet N(e_j) = N(N(e_i)\bullet e_j + e_i\bullet N(e_j)-N(e_i\bullet e_j))
\end{align*}
\begin{align*}
N(e_i)\ast N(e_j) = N(N(e_i)\ast e_j + e_i\ast N(e_j)-N(e_i\ast e_j))
\end{align*}
yield
\begin{align*}
 \sum_{k=1}^n\sum_{p=1}^nN_i^kN_j^p\alpha_{kp}^{q} = \sum_{k=1}^n\sum_{p=1}^n\Bigg(N_i^k\alpha_{kj}^pN_p^q+N_j^k\alpha_{ik}^pN_p^q-\alpha_{ij}^kN_k^pN_p^q\Bigg), 
\end{align*}
\begin{align*}
 \sum_{k=1}^n\sum_{p=1}^nN_i^kN_j^p\beta_{kp}^{q} = \sum_{k=1}^n\sum_{p=1}^n\Bigg(N_i^k\beta_{kj}^pN_p^q+N_j^k\beta_{ik}^pN_p^q-\beta_{ij}^kN_k^pN_p^q\Bigg), 
\end{align*}
for $i,j,q=1,2,\dots,n$.

We use a similar process to classify the other invariants. Let $\xi$ be a averaging operator on $\A$ with \[\chi(e_i) =\chi_{i}^1e_1 +\chi_{i}^2e_2 + \cdots +\chi_{i}^ne_n\] for $ \chi_i^1, \chi_i^2,\dots, \chi_i^n\in\C$. The equalities
\begin{align*}
\chi(\chi(e_i)\bullet e_j)=\chi(e_i)\bullet \chi(e_j)=\chi(e_i\bullet \chi(e_j)))
\end{align*}
\begin{align*}
\chi(\chi(e_i)\ast e_j)=\chi(e_i)\ast \chi(e_j)=\chi(e_i\ast \chi(e_j)))
\end{align*}
yield
\begin{align*}
 \sum_{k=1}^n\sum_{p=1}^n\chi_i^k\alpha_{kj}^{p}\chi_p^q = \sum_{k=1}^n\sum_{p=1}^n\chi_i^k\chi_j^p\alpha_{kp}^q=\sum_{k=1}^n\sum_{p=1}^n\chi_j^k\alpha_{ik}^p\chi_p^q, 
\end{align*}
\begin{align*}
 \sum_{k=1}^n\sum_{p=1}^n\chi_i^k\beta_{kj}^{p}\chi_p^q = \sum_{k=1}^n\sum_{p=1}^n\chi_i^k\chi_j^p\beta_{kp}^q=\sum_{k=1}^n\sum_{p=1}^n\chi_j^k\beta_{ik}^p\chi_p^q, 
\end{align*}
for $i,j,q=1,2,\dots,n$.

 Let $\xi$ be a averaging operator on $\A$ with \[\xi(e_i) =\xi_{i}^1e_1 +\xi_{i}^2e_2 + \cdots +\xi_{i}^ne_n\] for $\xi_i^1, \xi_i^2,\dots, \xi_i^n\in\C$. The equalities
\begin{align*}
\xi(e_i\bullet e_j)=\xi(\xi(e_i)\bullet e_j+e_i\bullet \xi(e_j)-\xi(\xi(e_i)\bullet \xi(e_j))
\end{align*}
\begin{align*}
\xi(e_i\ast e_j)=\xi(\xi(e_i)\ast e_j+e_i\ast \xi(e_j)-\xi(\xi(e_i)\ast \xi(e_j))
\end{align*}
yield
\begin{align*}
 \sum_{q=1}^n\alpha_{ij}^{q}\xi_q^r- \sum_{q=1}^n\sum_{p=1}^n\xi_i^p\alpha_{pj}^q\xi_q^r-\sum_{q=1}^n\sum_{p=1}^n\xi_j^p\alpha_{ip}^q\xi_q^r+\sum_{q=1}^n\sum_{p=1}^n\sum_{k=1}^n\xi_i^k\xi_j^p\alpha_{kp}^q\xi_q^r, 
\end{align*}
\begin{align*}
 \sum_{q=1}^n\beta_{ij}^{q}\xi_q^r- \sum_{q=1}^n\sum_{p=1}^n\xi_i^p\beta_{pj}^q\xi_q^r-\sum_{q=1}^n\sum_{p=1}^n\xi_j^p\beta_{ip}^q\xi_q^r+\sum_{q=1}^n\sum_{p=1}^n\sum_{k=1}^n\xi_i^k\xi_j^p\beta_{kp}^q\xi_q^r, 
\end{align*}
for $i,j,q=1,2,\dots,n$.

For the remaining operators, we solve similar systems of equations and present our results using the following structure constants.

For the linear maps $d'$ and $d''$ that appear in the definitions of quasi-derivations and generalized derivations, we will use the notations \begin{align*}
    d'(e_i) = {d'}_i^1e_1 + {d'}_i^2e_2 + \cdots + {d'}_i^ne_n, && d''(e_i) = {d''}_i^1e_1 + {d''}_i^2e_2 + \cdots + {d''}_i^ne_n
\end{align*} for their structure constants.

\begin{prop}
The following matrix forms characterize the derivations, centroids, quasi-centroids, Rota-Baxter operators, Nijenhuis operators, averaging operators, Reynolds operators, quasi-derivations, and generalized derivations (respectively) of  $2$-dimensional compatible associative algebras.

$(\A^{1}_2,\A^{4}_2)$ :
$\left(
\begin{array}{cccc}
0&0\\
0&d_{2}^2
\end{array}
\right)$
;
$\left(\begin{array}{cccc}
\eta_{1}^1&0\\
0&\eta_{2}^1
\end{array}
\right)$
;
$\left(\begin{array}{cccc}
\delta_{1}^1&0\\
\delta_{1}^2&\delta_{1}^2
\end{array}
\right)$
;
$\left(\begin{array}{cccc}
0&0\\
R_{1}^2&0
\end{array}
\right)$
;
$\left(\begin{array}{cccc}
0&0\\
0&N_{2}^2
\end{array}
\right)$
;
$\left(\begin{array}{cccc}
\xi^1_1&0\\
0&\xi_{2}^2
\end{array}
\right)$
;
$\left(\begin{array}{cccc}
0&0\\
\chi^2_1&0
\end{array}
\right)$
;
\\
$\left\{\left(\begin{array}{cccc}
d^1_1&0\\
{d'}_{1}^2&{d}_{2}^2
\end{array}
\right),
\left(\begin{array}{cccc}
{d'}^1_1&{d'}_{1}^{2}\\
{d'}_{1}^2+2{d'}_{2}^2&d_{2}^2
\end{array}
\right)	\right\}$
;
\\
$\left\{\left(\begin{array}{cccc}
d^1_1&-{d''}_{1}^1+{d'}_{1}^1-{d}_{1}^1\\
{d''}_{1}^2&{d''}_{2}^2
\end{array}
\right),
\left(\begin{array}{cccc}
{d}^1_1&{d}_{2}^{1}\\
0&{d''}_{1}^2+{d''}_{2}^{2}-{d'}^{1}_1
\end{array}
\right),
\left(\begin{array}{cccc}
{d'}^1_1&0\\
-{d}_{1}^2&{d''}_{1}^2+{d''}_2^2-d^1_1
\end{array}
\right)\right\}$.
\end{prop}

\begin{prop}
The following matrix forms characterize the derivations, centroids, quasi-centroids, Rota-Baxter operators, Nijenhuis operators, averaging operators, and Reynolds operators (respectively) of $3$-dimensional compatible associative algebras.
\begin{enumerate}
	\item[$(\A^{1}_3,\A^{2}_3)$] : 
$\left(\begin{array}{cccc}
d^1_1&0&0\\
d^2_1&d_2^2&d_3^2\\
0&0&d_2^3-d_1^3
\end{array}
\right)$
,
$\left(\begin{array}{cccc}
\eta_{2}^1&0&0\\
\eta_{1}^2&\beta_{2}^2&\beta_{3}^2\\
0&0&\eta_{2}^3
\end{array}
\right)$
,
$\left(\begin{array}{cccc}
\delta_{1}^1&0&0\\
\delta_{1}^2&\delta_{2}^2&\delta_{3}^2\\
0&0&\delta_{1}^3
\end{array}
\right)$
,
$\left(\begin{array}{cccc}
0&0&0\\
R_1^2&R_2^2&R_3^2\\
0&0&0
\end{array}
\right)$
,
$\left(\begin{array}{cccc}
0&0&0\\
N^2_1&N_2^2&N_3^2\\
0&0&N_2^3
\end{array}
\right)$
,
$\left(\begin{array}{cccc}
\xi^1_3&0&0\\
0&\xi_3^2&\xi_3^2\\
0&0&\xi_3^3
\end{array}
\right)$
,
$\left(\begin{array}{cccc}
0&0&0\\
\chi_1^2&\chi_2^2&\chi_3^2\\
0&0&0
\end{array}
\right)$.
\item[$(\A^{1}_3,\A^{11}_3)$]  : 
$\left(\begin{array}{cccc}
d_{1}^1&0&0\\
0&d_{1}^2&0\\
0&0&0
\end{array}
\right)$
,
$\left(\begin{array}{cccc}
\eta_{1}^1&0&0\\
0&\eta_{1}^2&\beta_{3}^2\\
0&0&\eta_{1}^3
\end{array}
\right)$
,
$\left(\begin{array}{cccc}
\delta_{1}^1&0&\delta_{3}^1\\
0&\delta_{1}^2&\delta_{3}^2\\
0&0&\delta_{1}^3
\end{array}
\right)$
,
$\left(\begin{array}{cccc}
0&0&0\\
R_1^2&0&R_3^2\\
0&0&0
\end{array}
\right)$
,
$\left(\begin{array}{cccc}
N_2^1&0&0\\
0&N_2^2&0\\
0&0&0
\end{array}
\right)$
,
$\left(\begin{array}{cccc}
\xi^1_3&0&0\\
0&\xi_3^2&\xi_3^2\\
0&0&\xi_3^3
\end{array}
\right)$
,
$\left(\begin{array}{cccc}
0&0&\chi_3^1\\
0&0&\chi_3^2\\
0&0&0
\end{array}
\right)$.
\item[$(\A^{10}_3,\A^{12}_3)$]  : 
$\left(\begin{array}{cccc}
d_{1}^1&0&0\\
d^2_1&2d_{1}^2&0\\
0&0&0
\end{array}
\right)$
,
$\left(\begin{array}{cccc}
\eta_{1}^1&0&0\\
0&\eta_{1}^2&\eta_{3}^2\\
0&0&\eta_{1}^3
\end{array}
\right)$
,
$\left(\begin{array}{cccc}
\delta_{1}^1&0&0\\
0&\delta_{1}^2&\delta_{3}^2\\
0&0&\delta_{1}^3
\end{array}
\right)$
,
$\left(\begin{array}{cccc}
0&0&0\\
R_1^2&0&R_3^2\\
0&0&0
\end{array}
\right)$
,
$\left(\begin{array}{cccc}
\frac{N_2^1}{2}&0&0\\
N^2_1&N_2^2&0\\
0&0&0
\end{array}
\right)$
,
$\left(\begin{array}{cccc}
\xi^1_3&0&0\\
0&\xi_3^2&\xi_3^2\\
0&0&\xi_3^3
\end{array}
\right)$
,
\,
$\left(\begin{array}{cccc}
\chi_2^1&0&\chi_3^1\\
0&\chi_2^2&\chi_3^2\\
0&0&\chi_2^3
\end{array}
\right)$.
\end{enumerate}
\end{prop}

\begin{prop}
The following matrix forms characterize the derivations, centroids and quasi-centroids (respectively) of $4$-dimensional compatible associative algebras.
\begin{enumerate}
    \item[$(\A^{1}_4,\A^{3}_4)$] :   
$\left(\begin{array}{cccc}
d^1_1&0&0&0\\
0&d^2_1&0&0\\
d^3_1&d^3_2&d_1^3&0\\
d^4_1&d^4_2&0&2d_1^4
\end{array}
\right)$
,
\quad
$\left(\begin{array}{cccc}
\eta_{2}^1&0&0&0\\
0&\eta_{1}^2&0&0\\
\eta_{1}^3&\eta_{2}^3&\eta_{1}^3&0\\
\eta_{1}^4&\eta_{2}^4&0&\eta_{1}^4
\end{array}
\right)$
,
\quad
$\left(\begin{array}{cccc}
\delta_{1}^1&0&0&0\\
0&0&\delta_{1}^2&0\\
\delta_{1}^3&\delta_{2}^3&\delta_{3}^3&\delta_{4}^3\\
\delta_{1}^4&\delta_{2}^4&0&(\delta_{2}^4)^2\\
\end{array}
\right)$

  \item[$(\A^{1}_4,\A^{4}_4)$] : 
$\left(\begin{array}{cccc}
0&0&0&0\\
0&0&0&0\\
0&0&0&0\\
d^4_1&d^4_2&0&0
\end{array}
\right)$
,
\quad
$\left(\begin{array}{cccc}
\eta_{1}^1&0&0&0\\
0&\eta_{1}^2&0&0\\
0&0&\eta_{1}^3&0\\
\eta_{1}^4&\eta_{2}^4&0&\eta_{1}^4
\end{array}
\right)$
,
\quad
$\left(\begin{array}{cccc}
\delta_{1}^1&0&0&0\\
0&\delta_{1}^2&0&0\\
0&0&\delta_{1}^3&0\\
\delta_{1}^4&\delta_{2}^4&\delta_{3}^4&\delta_{4}^4\\
\end{array}
\right)$
  \item[$(\A^{1}_4,\A^{6}_4)$] : 
$\left(\begin{array}{cccc}
0&0&0&0\\
0&0&0&0\\
0&0&0&0\\
d^4_1&d^4_2&0&0
\end{array}
\right)$
,
\quad
$\left(\begin{array}{cccc}
\eta_{4}^1&0&0&0\\
0&\eta_{4}^2&0&0\\
0&0&\eta_{4}^3&0\\
\eta_{1}^4&\eta_{2}^4&0&\eta_{4}^4
\end{array}
\right)$
,
\quad
$\left(\begin{array}{cccc}
\delta_{1}^1&0&0&0\\
0&\delta_{1}^2&0&0\\
0&0&\delta_{1}^3&0\\
\delta_{1}^4&\delta_{2}^4&\delta_{3}^4&\delta_{4}^4\\
\end{array}
\right)$

 \item[$(\A^{1}_4,\A^{17}_4)$] :
$\left(\begin{array}{cccc}
d^1_1&0&0&0\\
0&d^2_1&0&0\\
d^3_1&d^3_2&2d^3_1&0\\
d^4_1&d^4_2&0&2d^4_1
\end{array}
\right)$
,
\quad
$\left(\begin{array}{cccc}
\eta_{1}^1&0&0&0\\
0&\eta_{1}^2&0&0\\
\beta_{1}^3&\eta_{2}^3&\eta_{1}^3&0\\
\beta_{1}^4&\eta_{2}^4&0&\eta_{1}^4
\end{array}
\right)$
,
\quad
$\left(\begin{array}{cccc}
\delta_{1}^1&0&0&0\\
0&\delta_{1}^2&0&0\\
\delta_{1}^3&\delta_{2}^3&\delta_{3}^3&0\\
\delta_{1}^4&\delta_{2}^4&\delta_{3}^4&\delta_{4}^4\\
\end{array}
\right)$

 \item[$(\A^{1}_4,\A^{18}_4)$] :
$\left(\begin{array}{cccc}
d^1_1&0&0&0\\
0&d^2_1&0&0\\
d^3_1&d^3_2&2d^3_1&0\\
d^4_1&d^4_2&0&2d^4_1
\end{array}
\right)$
,
\quad
$\left(\begin{array}{cccc}
\eta_{1}^1&0&0&0\\
0&\eta_{1}^2&0&0\\
\eta_{1}^3&\eta_{2}^3&\eta_{1}^3&0\\
\eta_{1}^4&\eta_{2}^4&0&\eta_{1}^4
\end{array}
\right)$
,
\quad
$\left(\begin{array}{cccc}
\delta_{1}^1&0&0&0\\
0&\delta_{1}^2&0&0\\
\delta_{1}^3&\delta_{2}^3&\delta_{3}^3&0\\
\delta_{1}^4&\delta_{2}^4&\delta_{3}^4&\delta_{4}^4\\
\end{array}
\right)$

 \item[$(\A^{1}_4,\A^{24}_4)$] :
$\left(\begin{array}{cccc}
0&0&0&0\\
0&0&0&0\\
0&d^3_2&0&0\\
0&0&d^4_2&d^4_4
\end{array}
\right)$
,
\quad
$\left(\begin{array}{cccc}
\eta_{1}^1&0&0&0\\
0&\eta_{1}^2&0&0\\
0&0&\eta_{1}^3&0\\
0&0&0&\eta_{1}^4
\end{array}
\right)$
,
\quad
$\left(\begin{array}{cccc}
\delta_{1}^1&\delta_{2}^1&0&0\\
0&\delta_{1}^2&0&0\\
0&0&\delta_{1}^3&0\\
0&0&0&\delta_{1}^4\\
\end{array}
\right)$

\item[$(\A^{2}_4,\A^{5}_4)$] :
$\left(\begin{array}{cccc}
d^1_1&0&0&0\\
0&d_1^2&0&0\\
0&0&d_1^3&0\\
d_1^4&d_2^4&d^4_3&d^4_4
\end{array}
\right)$
,
\quad
$\left(\begin{array}{cccc}
\eta_{1}^1&0&0&0\\
0&\eta_{1}^2&0&0\\
0&0&\eta_{1}^3&0\\
\eta_{1}^4&\eta_{2}^4&\eta_{3}^4&\eta_{1}^4
\end{array}
\right)$
,
\quad
$\left(\begin{array}{cccc}
\delta_{1}^1&0&0&0\\
0&\delta_{1}^2&0&0\\
0&0&\delta_{1}^3&0\\
\delta_{1}^4&\delta_{2}^4&\delta_{3}^4&\delta_{1}^4\\
\end{array}
\right)$

\item[$(\A^{2}_4,\A^{10}_4)$] :
$\left(\begin{array}{cccc}
0&0&0&0\\
0&0&0&0\\
0&0&0&0\\
d_1^4&d_2^4&0&0
\end{array}
\right)$
,
\quad
$\left(\begin{array}{cccc}
\eta_{1}^1&0&0&0\\
0&\eta_{1}^2&0&0\\
0&0&\eta_{1}^3&0\\
\beta_{1}^4&\eta_{2}^4&0&\eta_{1}^4
\end{array}
\right)$
,
\quad
$\left(\begin{array}{cccc}
\delta_{1}^1&0&0&0\\
0&\delta_{1}^2&0&0\\
0&0&\delta_{1}^3&0\\
\delta_{1}^4&\delta_{2}^4&\delta_{3}^4&\delta_{1}^4\\
\end{array}
\right)$

\item[$(\A^{2}_4,\A^{16}_4)$] :
$\left(\begin{array}{cccc}
d^1_1&0&0&0\\
0&d_1^2&0&0\\
0&0&d_1^3&0\\
d_1^4&d_2^4&d^4_3&2d^4_1
\end{array}
\right)$
,
\quad
$\left(\begin{array}{cccc}
\eta_{4}^1&0&0&0\\
0&\eta_{4}^2&0&0\\
0&0&\eta_{4}^3&0\\
\beta_{1}^4&\eta_{2}^4&\eta_{3}^4&\eta_{4}^4
\end{array}
\right)$
,
\quad
$\left(\begin{array}{cccc}
\delta_{1}^1&0&0&0\\
0&\delta_{1}^2&0&0\\
0&0&\delta_{1}^3&0\\
\delta_{1}^4&\delta_{2}^4&\delta_{3}^4&\delta_{4}^4\\
\end{array}
\right)$

\item[$(\A^{2}_4,\A^{25}_4)$] :
$\left(\begin{array}{cccc}
d^1_1&0&0&0\\
0&d_1^2&0&0\\
0&0&d_1^3&0\\
d_1^4&d_2^4&d^4_3&2d^4_1
\end{array}
\right)$
,
\quad
$\left(\begin{array}{cccc}
\eta_{1}^1&0&0&0\\
0&\eta_{1}^2&0&0\\
0&0&\eta_{1}^3&0\\
\eta_{1}^4&\eta_{2}^4&\eta_{3}^4&\eta_{1}^4
\end{array}
\right)$
,
\quad
$\left(\begin{array}{cccc}
\delta_{1}^1&0&0&0\\
0&\delta_{1}^2&0&0\\
0&0&\delta_{1}^3&0\\
\delta_{1}^4&\delta_{2}^4&\delta_{3}^4&\delta_{4}^4\\
\end{array}
\right)$

\item[$(\A^{2}_4,\A^{28}_4)$] :
$\left(\begin{array}{cccc}
0&0&0&0\\
0&d^2_2&0&0\\
0&0&d^3_2&0\\
d_1^4&0&0&d^4_2
\end{array}
\right)$
,
\quad
$\left(\begin{array}{cccc}
\eta_{1}^1&0&0&0\\
0&\eta_{1}^2&0&0\\
0&0&\eta_{1}^3&0\\
\eta_{1}^4&\eta_{2}^4&\eta_{3}^4&\eta_{1}^4
\end{array}
\right)$
,
\quad
$\left(\begin{array}{cccc}
\delta_{1}^1&0&0&0\\
0&\delta_{1}^2&0&0\\
0&0&\delta_{1}^3&0\\
0&0&0&\delta_{1}^4\\
\end{array}
\right)$

\item[$(\A^{3}_4,\A^{4}_4)$] :
$\left(\begin{array}{cccc}
d^1_1&0&0&0\\
0&d^2_1&0&0\\
d^3_2&d^2_1&2d^3_1&0\\
d_1^4&d^4_2&0&2d^4_1
\end{array}
\right)$
,
\quad
$\left(\begin{array}{cccc}
\eta_{1}^1&0&0&0\\
0&\eta_{1}^2&0&0\\
\eta_{1}^3&\eta_{2}^3&\eta_{1}^3&0\\
\eta_{1}^4&\eta_{2}^4&0&\eta_{1}^4
\end{array}
\right)$
,
\quad
$\left(\begin{array}{cccc}
\delta_{1}^1&0&0&0\\
0&\delta_{1}^2&0&0\\
\delta_{1}^3&\delta_{2}^3&\delta_{3}^3&\delta_{3}^4\\
\delta_{1}^4&\delta_{2}^4&0&\delta_{1}^4\\
\end{array}
\right)$

\item[$(\A^{3}_4,\A^{6}_4)$] :
$\left(\begin{array}{cccc}
d^1_2&0&0&0\\
d^2_1&d^2_2&0&0\\
d^3_2&d^2_1&2d^3_2&0\\
d_1^4&d^4_2&0&2d^4_2
\end{array}
\right)$
,
\quad
$\left(\begin{array}{cccc}
\eta_{4}^1&0&0&0\\
0&\eta_{4}^2&0&0\\
\eta_{1}^3&\eta_{2}^3&\eta_{4}^3&0\\
\eta_{1}^4&\eta_{2}^4&0&\eta_{4}^4
\end{array}
\right)$
,
\quad
$\left(\begin{array}{cccc}
\delta_{1}^1&0&0&0\\
0&\delta_{1}^2&0&0\\
\delta_{1}^3&\delta_{2}^3&\delta_{3}^3&\delta_{3}^4\\
\delta_{1}^4&\delta_{2}^4&0&\delta_{4}^4\\
\end{array}
\right)$

\item[$(\A^{3}_4,\A^{17}_4)$] :
$\left(\begin{array}{cccc}
d^1_2&0&0&0\\
01&d^2_1&0&0\\
d^3_2&d^2_1&2d^3_1&0\\
d_1^4&d^4_2&0&2d^4_1
\end{array}
\right)$
,
\quad
$\left(\begin{array}{cccc}
\eta_{1}^1&0&0&0\\
0&\eta_{1}^2&0&0\\
\eta_{1}^3&0&\eta_{1}^3&0\\
\eta_{1}^4&\eta_{2}^4&0&\eta_{1}^4
\end{array}
\right)$
,
\quad
$\left(\begin{array}{cccc}
\delta_{1}^1&0&0&0\\
0&\delta_{1}^2&0&0\\
\delta_{1}^3&\delta_{2}^3&\delta_{3}^3&\delta_{4}^3\\
\delta_{1}^4&\delta_{2}^4&\delta_{3}^4&\delta_{4}^4\\
\end{array}
\right)$

\item[$(\A^{3}_4,\A^{18}_4)$] : 
$\left(\begin{array}{cccc}
d^1_2&0&0&0\\
0&d^2_1&0&0\\
d^3_2&d^2_1&2d^3_1&0\\
d_1^4&d^4_2&0&2d^4_1
\end{array}
\right)$
,
\quad
$\left(\begin{array}{cccc}
\eta_{1}^1&0&0&0\\
0&\eta_{1}^2&0&0\\
\eta_{1}^3&\eta_{2}^3&\eta_{1}^3&0\\
\eta_{1}^4&\eta_{2}^4&0&\eta_{1}^4
\end{array}
\right)$
,
\quad
$\left(\begin{array}{cccc}
\delta_{1}^1&0&0&0\\
0&\delta_{1}^2&0&0\\
\delta_{1}^3&\delta_{2}^3&\delta_{3}^3&\delta_{4}^3\\
\delta_{1}^4&\delta_{2}^4&\delta_{3}^4&\delta_{4}^4\\
\end{array}
\right)$

\item[$(\A^{4}_4,\A^{6}_4)$] :
$\left(\begin{array}{cccc}
d^1_2&0&0&0\\
0&d^2_2&0&0\\
d^3_1&d^2_3&2d^3_2&0\\
d_1^4&d^4_2&0&2d^4_2
\end{array}
\right)$
,
\quad
$\left(\begin{array}{cccc}
\eta_{4}^1&0&0&0\\
0&\eta_{4}^2&0&0\\
\eta_{1}^3&\eta_{2}^3&\eta_{4}^3&0\\
\eta_{1}^4&\eta_{2}^4&0&\eta_{4}^4
\end{array}
\right)$
,
\quad
$\left(\begin{array}{cccc}
\delta_{1}^1&0&0&0\\
0&\delta_{1}^2&0&0\\
\delta_{1}^3&\delta_{2}^3&\delta_{3}^3&\delta_{4}^3\\
\delta_{1}^4&\delta_{2}^4&\delta_{3}^4&\delta_{4}^4\\
\end{array}
\right)$

\item[$(\A^{4}_4,\A^{17}_4)$] : 
$\left(\begin{array}{cccc}
d^1_1&0&0&0\\
0&d^2_1&0&0\\
d^3_1&d^2_3&2d^3_1&0\\
d_1^4&d^4_2&0&2d^4_1
\end{array}
\right)$
,
\quad
$\left(\begin{array}{cccc}
\eta_{1}^1&0&0&0\\
0&\eta_{1}^2&0&0\\
\eta_{1}^3&\eta_{2}^3&\eta_{1}^3&0\\
\eta_{1}^4&\eta_{2}^4&0&\eta_{1}^4
\end{array}
\right)$
,
\quad
$\left(\begin{array}{cccc}
\delta_{1}^1&0&0&0\\
0&\delta_{1}^2&0&0\\
\delta_{1}^3&\delta_{2}^3&\delta_{3}^3&\delta_{4}^3\\
\delta_{1}^4&\delta_{2}^4&\delta_{3}^4&\delta_{4}^4\\
\end{array}
\right)$

\item[$(\A^{4}_4,\A^{18}_4)$] : 
$\left(\begin{array}{cccc}
0&0&0&0\\
0&0&0&0\\
0&0&0&0\\
d_1^4&d^4_2&0&0
\end{array}
\right)$
,
\quad
$\left(\begin{array}{cccc}
\eta_{4}^1&0&0&0\\
0&\eta_{4}^2&0&0\\
0&0&\eta_{4}^3&0\\
\eta_{1}^4&\eta_{2}^4&0&\eta_{4}^4
\end{array}
\right)$
,
\quad
$\left(\begin{array}{cccc}
\delta_{1}^1&0&0&0\\
0&\delta_{1}^2&0&0\\
0&0&\delta_{3}^3&0\\
\delta_{1}^4&\delta_{2}^4&\delta_{3}^4&\delta_{4}^4\\
\end{array}
\right)$

\item[$(\A^{5}_4,\A^{25}_4)$] :
$\left(\begin{array}{cccc}
d_1^1&0&0&0\\
0&d_1^2&0&0\\
0&0&d_1^3&0\\
d_1^4&d^4_2&d_3^4&2d_1^4
\end{array}
\right)$
,
\quad
$\left(\begin{array}{cccc}
\eta_{1}^1&0&0&0\\
0&\eta_{1}^2&0&0\\
0&0&\eta_{1}^3&0\\
\eta_{1}^4&\eta_{2}^4&0&\eta_{1}^4
\end{array}
\right)$
,
\quad
$\left(\begin{array}{cccc}
\delta_{1}^1&0&0&0\\
0&\delta_{1}^2&0&0\\
0&0&\delta_{3}^3&0\\
\delta_{1}^4&\delta_{2}^4&\delta_{3}^4&\delta_{4}^4\\
\end{array}
\right)$

\item[$(\A^{5}_4,\A^{28}_4)$] :
$\left(\begin{array}{cccc}
\frac{1}{2}d_4^1&0&0&0\\
d_4^2&\frac{1}{2}d_4^2&0&0\\
-d_1^2&0&\frac{1}{2}d_4^3&0\\
d_1^4&d^4_2&d_3^4&d_4^4
\end{array}
\right)$
,
\quad
$\left(\begin{array}{cccc}
\eta_{4}^1&0&0&0\\
0&\eta_{4}^2&0&0\\
0&0&\eta_{4}^3&0\\
\eta_{1}^4&\eta_{2}^4&\eta_{3}^4&\eta_{4}^4
\end{array}
\right)$
,
\quad
$\left(\begin{array}{cccc}
\delta_{1}^1&0&0&0\\
0&\delta_{1}^2&0&0\\
0&0&\delta_{3}^3&0\\
\delta_{1}^4&\delta_{2}^4&\delta_{3}^4&\delta_{4}^4\\
\end{array}
\right)$

\item[$(\A^{6}_4,\A^{17}_4)$] : 
$\left(\begin{array}{cccc}
d_2^1&0&0&0\\
0&d_2^2&0&0\\
d_3^1&d_2^3&2d_2^3&0\\
d_1^4&d^4_2&0&2d_2^4
\end{array}
\right)$
,
\quad
$\left(\begin{array}{cccc}
\eta_{4}^1&0&0&0\\
0&\eta_{4}^2&0&0\\
0&0&\eta_{4}^3&0\\
\eta_{1}^4&\eta_{2}^4&0&\eta_{4}^4
\end{array}
\right)$
,
\quad
$\left(\begin{array}{cccc}
\delta_{1}^1&0&0&0\\
0&\delta_{1}^2&0&0\\
\delta_{1}^3&\delta_{2}^3&\delta_{3}^3&\delta_{4}^3\\
\delta_{1}^4&\delta_{2}^4&\delta_{3}^4&\delta_{4}^4\\
\end{array}
\right)$

\item[$(\A^{6}_4,\A^{18}_4)$] :
$\left(\begin{array}{cccc}
0&0&0&0\\
0&0&0&0\\
0&0&0&0\\
d_1^4&d^4_2&0&0
\end{array}
\right)$
,
\quad
$\left(\begin{array}{cccc}
\eta_{1}^1&0&0&0\\
0&\eta_{1}^2&0&0\\
0&0&\eta_{1}^3&0\\
\eta_{1}^4&\eta_{2}^4&0&\eta_{1}^4
\end{array}
\right)$
,
\quad
$\left(\begin{array}{cccc}
\delta_{1}^1&0&0&0\\
0&\delta_{1}^2&0&0\\
\delta_{1}^3&0&\delta_{1}^3&0\\
\delta_{1}^4&\delta_{2}^4&\delta_{3}^4&\delta_{4}^4\\
\end{array}
\right)$

\item[$(\A^{10}_4,\A^{25}_4)$] :
$\left(\begin{array}{cccc}
d_2^1&0&0&0\\
0&d_2^2&0&0\\
d_3^1&d_2^3&2d_2^3&0\\
d_1^4&d^4_2&0&2d_2^4
\end{array}
\right)$
,
\quad
$\left(\begin{array}{cccc}
\eta_{4}^1&0&0&0\\
0&\eta_{4}^2&0&0\\
0&0&\eta_{4}^3&0\\
\eta_{1}^4&\eta_{2}^4&0&\eta_{4}^4
\end{array}
\right)$
,
\quad
$\left(\begin{array}{cccc}
\delta_{1}^1&0&0&0\\
0&\delta_{1}^2&0&0\\
\delta_{1}^3&\delta_{2}^3&\delta_{3}^3&\delta_{4}^3\\
\delta_{1}^4&\delta_{2}^4&\delta_{3}^4&\delta_{4}^4\\
\end{array}
\right)$

\item[$(\A^{11}_4,\A^{23}_4)$] :
$\left(\begin{array}{cccc}
0&0&0&0\\
0&0&0&0\\
d_1^3&d^3_2&0&d^3_4\\
-\frac{1}{2}d_3^4&0&0&0
\end{array}
\right)$
,
\quad
$\left(\begin{array}{cccc}
\eta_{1}^1&0&0&0\\
0&\eta_{1}^2&0&0\\
\eta_{1}^3&\eta_{2}^3&\eta_{1}^3&0\\
\eta_{1}^4&0&0&\eta_{1}^4
\end{array}
\right)$
,
\quad
$\left(\begin{array}{cccc}
\delta_{1}^1&0&0&0\\
0&\delta_{1}^2&0&0\\
\delta_{1}^3&\delta_{2}^3&\delta_{3}^3&\delta_{4}^3\\
\delta_{1}^4&\delta_{2}^4&\delta_{3}^4&\delta_{4}^4\\
\end{array}
\right)$

\item[$(\A^{16}_4,\A^{25}_4)$] : 
$\left(\begin{array}{cccc}
\frac{1}{2}d_4^1&0&0&0\\
0&\frac{1}{2}d_4^2&0&0\\
0&0&\frac{1}{2}d_4^3&0\\
d_1^4&d^4_1&d_3^4&d^4_4\\
\end{array}
\right)$
,
\quad
$\left(\begin{array}{cccc}
\eta_{4}^1&0&0&0\\
0&\eta_{4}^2&0&0\\
0&0&\eta_{4}^3&0\\
\eta_{1}^4&\eta_{2}^4&\eta_{3}^4&\eta_{4}^4
\end{array}
\right)$
,
\quad
$\left(\begin{array}{cccc}
\delta_{1}^1&0&0&0\\
0&\delta_{1}^2&0&0\\
0&0&\delta_{1}^3&0\\
\delta_{1}^4&\delta_{2}^4&\delta_{3}^4&\delta_{4}^4\\
\end{array}
\right)$

\item[$(\A^{16}_4,\A^{28}_4)$] :
$\left(\begin{array}{cccc}
0&0&0&0\\
0&0&0&0\\
0&0&d^3_3&0\\
d_1^4&-d^4_1&0&0\
\end{array}
\right)$
,
\quad
$\left(\begin{array}{cccc}
\eta_{4}^1&0&0&0\\
0&\eta_{4}^2&0&0\\
0&0&\eta_{4}^3&0\\
0&0&0&\eta_{4}^4
\end{array}
\right)$
,
\quad
$\left(\begin{array}{cccc}
\delta_{1}^1&0&0&0\\
0&\delta_{1}^2&0&0\\
0&0&\delta_{1}^3&0\\
0&0&0&\delta_{1}^4\\
\end{array}
\right)$

\item[$(\A^{17}_4,\A^{18}_4)$] : 
$\left(\begin{array}{cccc}
d_2^1&0&0&0\\
0&d_2^2&0&0\\
d_1^3&d_2^3&2d^3_2&0\\
d_1^4&d^4_2&0&2d_2^4
\end{array}
\right)$
,
\quad
$\left(\begin{array}{cccc}
\eta_{3}^1&0&0&0\\
0&\eta_{3}^2&0&0\\
0&0&\eta_{3}^3&0\\
\eta_{1}^4&\eta_{2}^3&0&\eta_{4}^4
\end{array}
\right)$
,
\quad
$\left(\begin{array}{cccc}
\delta_{1}^1&0&0&0\\
0&\delta_{1}^2&0&0\\
\delta_{1}^2&\delta_{2}^3&\delta_{3}^3&0\\
\delta_{1}^4&\delta_{2}^4&\delta_{3}^4&\delta_{4}^4\\
\end{array}
\right)$

\item[$(\A^{25}_4,\A^{28}_4)$] : 
$\left(\begin{array}{cccc}
0&0&0&0\\
0&0&0&0\\
0&0&0&0\\
d_1^4&d^4_2&0&0
\end{array}
\right)$
,
\quad
$\left(\begin{array}{cccc}
\eta_{2}^1&0&0&0\\
0&\eta_{2}^2&0&0\\
0&0&\eta_{2}^3&0\\
\eta_{1}^4&\eta_{2}^3&0&\eta_{2}^4
\end{array}
\right)$
,
\quad
$\left(\begin{array}{cccc}
\delta_{1}^1&0&0&0\\
0&\delta_{1}^2&0&0\\
\delta_{1}^2&\delta_{2}^3&\delta_{3}^3&0\\
\delta_{1}^4&\delta_{2}^4&\delta_{3}^4&\delta_{4}^4\\
\end{array}
\right)$

\item[$(\A^{25}_4,\A^{33}_4)$] :
$\left(\begin{array}{cccc}
0&0&0&0\\
0&0&0&0\\
0&0&0&0\\
d_1^4&0&0&0
\end{array}
\right)$
,
\quad
$\left(\begin{array}{cccc}
\eta_{1}^1&0&0&0\\
0&\eta_{1}^2&0&0\\
0&0&\eta_{1}^3&0\\
0&0&0&\eta_{1}^4
\end{array}
\right)$
,
\quad
$\left(\begin{array}{cccc}
\delta_{1}^1&0&0&0\\
0&\delta_{1}^2&0&0\\
\delta_{1}^3&0&\delta_{1}^3&0\\
\delta_{1}^4&\delta_{2}^4&\delta_{3}^4&\delta_{4}^4\\
\end{array}
\right)$

\item[$(\A^{33}_4,\A^{45}_4)$] : 
$\left(\begin{array}{cccc}
0&0&0&0\\
0&0&0&0\\
0&0&0&0\\
0&0&0&0
\end{array}
\right)$
,
\quad
$\left(\begin{array}{cccc}
\eta_{1}^1&0&0&0\\
\eta_{1}^2&\eta_{1}^2&0&0\\
\eta_{1}^3&\eta_{2}^3&\eta_{1}^3&0\\
\eta_{1}^4&\eta_{2}^4&\eta_{3}^4&\eta_{1}^4
\end{array}
\right)$
,
\quad
$\left(\begin{array}{cccc}
\delta_{1}^1&0&0&0\\
\delta_{1}^2&\delta_{1}^2&0&0\\
\delta_{1}^3&\delta_{2}^3&\delta_{1}^3&0\\
\delta_{1}^4&\delta_{2}^4&\delta_{3}^4&\delta_{4}^4\\
\end{array}
\right)$

\item[$(\A^{37}_4,\A^{39}_4)$] : 
$\left(\begin{array}{cccc}
0&0&0&0\\
0&d_2^2&0&0\\
0&0&2d_2^3&0\\
d_1^4&0&0&3d_2^4
\end{array}
\right)$
,
\quad
$\left(\begin{array}{cccc}
\eta_{1}^1&0&0&0\\
0&\eta_{1}^2&0&0\\
0&0&\eta_{1}^3&0\\
0&0&0&\eta_{1}^4
\end{array}
\right)$
,
\quad
$\left(\begin{array}{cccc}
\delta_{1}^1&0&0&0\\
\delta_{1}^2&\delta_{1}^2&0&0\\
\delta_{1}^3&\delta_{2}^3&\delta_{1}^3&0\\
\delta_{1}^4&\delta_{2}^4&\delta_{3}^4&\delta_{4}^4\\
\end{array}
\right)$

\item[$(\A^{40}_4,\A^{43}_4)$] : 
$\left(\begin{array}{cccc}
0&0&0&0\\
0&d_2^2&0&0\\
0&d_2^3&d_3^3&0\\
0&d_2^4&d_3^4&2d_2^2
\end{array}
\right)$
,
\quad
$\left(\begin{array}{cccc}
\eta_{1}^1&0&0&0\\
0&\eta_{1}^2&0&0\\
0&0&\eta_{1}^3&0\\
\eta_{1}^4&0&0&\eta_{1}^4
\end{array}
\right)$
,
\quad
$\left(\begin{array}{cccc}
\delta_{1}^1&0&0&0\\
0&\delta_{1}^2&0&0\\
0&0&\delta_{1}^3&0\\
0&0&0&\delta_{1}^4\\
\end{array}
\right)$

\item[$(\A^{40}_4,\A^{44}_4)$] : 
$\left(\begin{array}{cccc}
0&0&0&0\\
0&d_2^2&0&0\\
0&0&d_3^3&0\\
0&d_2^4&d_3^4&d_2^2+d_3^4
\end{array}
\right)$
,
\quad
$\left(\begin{array}{cccc}
\eta_{1}^1&0&0&0\\
0&\eta_{1}^2&0&0\\
0&0&\eta_{1}^3&0\\
\eta_{1}^4&0&0&\eta_{1}^4
\end{array}
\right)$
,
\quad
$\left(\begin{array}{cccc}
\delta_{1}^1&0&0&0\\
0&\delta_{1}^2&0&0\\
0&0&\delta_{1}^3&0\\
\delta_{1}^4&0&0&\delta_{1}^4\\
\end{array}
\right)$

\item[$(\A^{40}_4,\A^{45}_4)$] : 
$\left(\begin{array}{cccc}
0&0&0&0\\
0&d_2^2&0&0\\
0&d_2^3&d_3^3&0\\
0&d_2^4&2d_2^4&3d_2^4
\end{array}
\right)$
,
\quad
$\left(\begin{array}{cccc}
\eta_{1}^1&0&0&0\\
0&\eta_{1}^2&0&0\\
0&0&\eta_{2}^3&0\\
\eta_{1}^4&0&0&\eta_{1}^4
\end{array}
\right)$
,
\quad
$\left(\begin{array}{cccc}
\delta_{1}^1&0&0&0\\
0&\delta_{1}^2&0&0\\
0&0&\delta_{1}^3&0\\
\delta_{1}^4&0&0&\delta_{1}^4\\
\end{array}
\right)$

\item[$(\A^{40}_4,\A^{46}_4)$] :
$\left(\begin{array}{cccc}
0&0&0&0\\
0&d_2^2&0&0\\
0&d_2^3&d_3^3&0\\
0&d_2^4&2d_2^4&2d_2^4
\end{array}
\right)$
,
\quad
$\left(\begin{array}{cccc}
\eta_{1}^1&0&0&0\\
0&\eta_{1}^2&0&0\\
0&0&\eta_{2}^3&0\\
0&\eta_{1}^4&0&\eta_{1}^4
\end{array}
\right)$
,
\quad
$\left(\begin{array}{cccc}
\delta_{1}^1&0&0&0\\
0&\delta_{1}^2&0&0\\
0&0&\delta_{1}^3&0\\
\delta_{1}^4&0&0&\delta_{1}^4\\
\end{array}
\right)$

\item[$(\A^{43}_4,\A^{44}_4)$] :
$\left(\begin{array}{cccc}
0&0&0&0\\
0&d_2^2&0&0\\
0&0&d_2^3&0\\
0&d_2^4&d_2^4&2d_2^4
\end{array}
\right)$
,
\quad
$\left(\begin{array}{cccc}
\eta_{4}^2&0&0&0\\
0&\eta_{4}^2&0&0\\
0&0&\eta_{4}^3&0\\
0&\eta_{1}^4&0&\eta_{1}^4
\end{array}
\right)$
,
\quad
$\left(\begin{array}{cccc}
\delta_{1}^1&0&0&0\\
0&\delta_{1}^2&0&0\\
0&0&\delta_{1}^3&0\\
\delta_{1}^4&0&0&\delta_{1}^4\\
\end{array}
\right)$

\item[$(\A^{43}_4,\A^{45}_4)$] : 
$\left(\begin{array}{cccc}
0&0&0&0\\
0&d_2^2&0&0\\
0&0&d_2^3&0\\
0&d_2^4&d_2^4&2d_2^4
\end{array}
\right)$
,
\quad
$\left(\begin{array}{cccc}
\eta_{4}^2&0&0&0\\
0&\eta_{4}^2&0&0\\
0&0&\eta_{4}^3&0\\
0&\eta_{1}^4&0&\eta_{1}^4
\end{array}
\right)$
,
\quad
$\left(\begin{array}{cccc}
\delta_{1}^1&0&0&0\\
0&\delta_{1}^2&0&0\\
0&0&\delta_{1}^3&0\\
\delta_{1}^4&0&0&\delta_{1}^4\\
\end{array}
\right)$

\item[$(\A^{43}_4,\A^{46}_4)$] : 
$\left(\begin{array}{cccc}
0&0&0&0\\
0&d_2^2&0&0\\
0&0&d_2^3&0\\
0&d_2^4&d_2^4&2d_2^4
\end{array}
\right)$
,
\quad
$\left(\begin{array}{cccc}
\eta_{4}^2&0&0&0\\
0&\eta_{4}^2&0&0\\
0&0&\eta_{4}^3&0\\
0&\eta_{1}^4&0&\eta_{1}^4
\end{array}
\right)$
,
\quad
$\left(\begin{array}{cccc}
\delta_{1}^1&0&0&0\\
0&\delta_{1}^2&0&0\\
0&0&\delta_{1}^3&0\\
\delta_{1}^4&0&0&\delta_{1}^4\\
\end{array}
\right)$

\item[$(\A^{44}_4,\A^{46}_4)$] : 
$\left(\begin{array}{cccc}
0&0&0&0\\
0&d_2^2&0&0\\
0&0&d_2^3&0\\
0&d_2^4&d_2^4&2d_2^4
\end{array}
\right)$
,
\quad
$\left(\begin{array}{cccc}
\eta_{4}^2&0&0&0\\
0&\eta_{4}^2&0&0\\
0&0&\eta_{4}^3&0\\
0&\eta_{1}^4&0&\eta_{1}^4
\end{array}
\right)$
,
\quad
$\left(\begin{array}{cccc}
\delta_{1}^1&0&0&0\\
0&\delta_{1}^2&0&0\\
0&0&\delta_{1}^3&0\\
\delta_{1}^4&0&0&\delta_{1}^4\\
\end{array}
\right)$
\end{enumerate}
\end{prop}

\section*{Conclusion}
In this paper, we classified compatible pairs of complex associative algebras with dimension less than four. For each type of algebra, we calculated and studied various algebraic properties, such as derivations, centroids, automorphism groups, and different linear operators like Rota–Baxter, Nijenhuis, averaging, and Reynolds operators. We also looked at quasi-derivations and generalized derivations to better understand the symmetries and structure of these algebras. These results not only improve our understanding of compatible associative algebras but also provide helpful tools for future studies in deformation theory, operad theory, and noncommutative geometry.

\textbf{Data Availability:}
No data were used to support this study.

\textbf{Conflict of Interests:}  
The authors declare that they have no conficts of interest.  

\textbf{Acknowledgment:}  
We thank the referee for the helpful comments and suggestions that contributed to improving this paper.

\end{document}